\documentclass[10pt, reqno]{amsart}
\usepackage{amsmath,amsfonts,amsthm,bm} % Math packages
\numberwithin{equation}{section}

\usepackage{amsfonts}
\usepackage{amssymb}
\usepackage{latexsym}

\usepackage{tikz}
\usepackage{graphicx,graphics,epsfig}
\usepackage{epstopdf}	

\usepackage[]{enumerate}
\usepackage{verbatim}
\usepackage{bbm}

\usepackage[]{enumerate}

\usepackage[sc]{mathpazo}          % Palatino with smallcaps as text font
\usepackage{eulervm}               % Euler math
\usepackage[scaled=0.86]{berasans} % Bera for san serif family
\usepackage[scaled=1]{inconsolata} % Inconsolata for fixed width
\usepackage[T1]{fontenc}

\usepackage[final]{hyperref}
\hypersetup{colorlinks=true, linkcolor=blue, anchorcolor=blue, citecolor=red, filecolor=blue, menucolor=blue, pagecolor=blue, urlcolor=blue}

\DeclareMathOperator{\sech}{sech}

\newcommand{\Bigabs}[1]{\Bigl\vert #1 \Bigr\vert}

\newcommand{\norm}[1]{\left\Vert #1 \right\Vert}

\newcommand{\sett}[1]{\left\{   #1   \right\}}

\newcommand{\C}{\mathbb{C}}
\newcommand{\N}{\mathbb{N}}
\newcommand{\Z}{\mathbb{Z}}
\newcommand{\R}{\mathbb{R}}

\newcommand{\angles}[1]{\langle #1 \rangle}

\DeclareMathOperator{\supp}{supp}

\newtheorem{theorem}{Theorem}

\newtheorem{lemma}{Lemma}

\theoremstyle{definition}

\theoremstyle{remark}
\newtheorem{remark}{Remark}

 \author[A. Esfahani]{Amin Esfahani}
\address{Department of Mathematics, Nazarbayev University, Nur-Sultan 010000, Kazakhstan}
 \email{saesfahani@gmail.com, amin.esfahani@nu.edu.kz}

\begin{document}
 \author[A. Tesfahun]{Achenef Tesfahun}
\address{Department of Mathematics, Nazarbayev University, Nur-Sultan 010000, Kazakhstan}
 \email{achenef.tesfahun@nu.edu.kz }

%\title{  
%	%	\footnotetext{2020 Mathematical subject classification:  ,,}
%	%	\footnotetext{:  ,,}
%}
%
%\date{}

\title[Well-posedness and analyticity for Boussinesq equation]{Well-posedness and analyticity of solutions for the sixth-order Boussinesq equation}
\begin{abstract}
Studied in this paper is the sixth-order Boussinesq equation. We extend the local well-posedness theory for this equation with quadratic and cubic nonlinearities to  the high dimensional case. In spite of having the ``bad'' fourth term $\Delta u$ in the equation, we derive some dispersive estimates leading to the existence of local solutions which also improves the previous results in the cubic case. In addition, we show persistence of spatial analyticity of solutions for the cubic nonlinearity.
\end{abstract}
\maketitle

\section{Introduction}
The Boussinesq equation
\begin{equation}\label{bad-bouss}
u_{tt}-u_{xx}-u_{xxxx}+(u^2)_{xx}=0
\end{equation}
was first derived from   Euler's equations of motion by 
	Joseph  Boussinesq in 1872 \cite{bouss} to describe the propagation of
	  small amplitude long waves on the surface of water. This equation is so-called the  \textit{bad}  Boussinesq equation  in the literature because of  $-u_{xxxx}$ causing the nonexistence of local solution of \eqref{bad-bouss} on the unbounded special domain $\R$.
By 
flipping the sign on $u_{xxxx}$, 		  \eqref{bad-bouss} can also be altered yielding the 	    \textit{good}  Boussinesq equation
	\begin{equation}\label{good-bouss}
		u_{tt}-u_{xx}+u_{xxxx}+(u^2)_{xx}=0.
	\end{equation}
Both above Boussinesq equations   appear not only in the study of the dynamics of thin inviscid layers with free surface but also in the study of nonlinear vibrations along a string, the shape-memory alloys, the propagation of waves in elastic rods and in the continuum limit of lattice dynamics or electromagnetic waves in nonlinear dielectric materials (\cite{turitsyn,Zakharov}). 

The mathematical study on the initial value problem associated with 
\eqref{good-bouss} has been well-developed in the past decades. We refer the reader to the results on the well-posedness in \cite{ghk,gebawitz,kishimoto,kist}, and the references therein.

Equations \eqref{bad-bouss} and \eqref{good-bouss} can be extended to the sixth-order Boussinesq equation
 \begin{equation}\label{sBeq}
	u_{tt}=u_{xx} -\beta   u_{xxxx}+ u_{xxxxxx}-  (u^2)_{xx},  
\end{equation}
where $\beta=\pm1$. Equation \eqref{sBeq} 
was derived  in \cite{cmv} to justify the original Boussinesq equation as a physical modeling of water waves in the shallow fluid layers,  and also in nonlinear atomic chains.  In addition, Maugin in \cite{maugin}  proposed \eqref{sBeq} to model  the nonlinear lattice dynamics in elastic crystals (see also \cite{daripahu}).

The initial value problem associated with \eqref{sBeq} with $u(x,0)=f(x)$ and $u_t(x,0)=g(x)$, and the general nonlinearity $\pm|u|^pu$ with $0<p<4$, was first studied in \cite{efw}. It was proved     the local well-posedness  of \eqref{sBeq} in
$L^2(\R)\times\dot{H}^{-3}(\R)$ by applying the
Strichartz-type  estimates. They also extended their results to $H^{1}(\R)\times L^2(\R)$ for the general case $p>0$. This in turn grants the global solution in $H^2(\R)$ with the defocusing nonlinearity $|u|^pu$, and the  focusing case $-|u|^pu$ with small initial data, by using the energy conservation of \eqref{sBeq}. Recently, Geba and Witz \cite{gebawitz} generalized the result of \cite{we2014} with the nonlinearity $|u|^2u$ into the case $|u|^{2p}u$ with $p\in\mathbb{N}$ and show the global well-posednes of \eqref{sBeq} in $H^s(\R)\times H^{s-2}(\R)$ with $2-2/(2p)<s<2$.
 Taking into account the initial data  $u(x,0)=f(x)$ and $u_t(x,0)=g(x)$, it was proved in \cite{esfahanifarah} by   using the suitable Bourgain-type space  that \eqref{sBeq} is locally well-posed in $H^s(\R)\times H^{s-2}(\R)$ for $s>-1/2$.  The local well-posedness results to \eqref{sBeq} were improved to $s>-3/4$ in \cite{wangesfahani} by the $[k;Z]$-multiplier norm method \cite{tao}. This is somehow natural well-posedness index due to presence of sixth-order derivative. Indeed, ignoring the fourth-order term of \eqref{sBeq}, the linearized equation can be decomposed into  $$u_{tt}-u_{xxxxxx}=(\partial_t-\partial_{xxx})(\partial_t+\partial_{xxx})u,$$ which can be considered as a coupled linear KdV-type equations. Replacing the sixth-order term with $ \varepsilon u_{xxxxxx}$, it was also showed in \cite{we2021} that the local solution  of \eqref{sBeq} with $\beta=1$ converges to that of \eqref{good-bouss} in $C([0,T];H^s(\R))$ with $s\geq0$ as $\varepsilon\to0$. 

Our interest here is in investigating the high-dimensional   case of    \eqref{sBeq}, that is,
\begin{equation}\label{nsB}
\begin{cases}
		u_{tt}=\Delta (u-\beta\Delta u+\Delta^2u+f(u)),  \\
		u(x,0)=u_0(x),\quad u_t(x,0)=u_1(x),\qquad x\in\R^n,
\end{cases}
\end{equation}
where we focus on the   quadratic nonlinearity $f(u)=\pm u^2$ and the cubic one $f(u)=\pm u^3$. As far as the authors know, this problem has not yet been studied. The only well-posedness result related to \eqref{nsB} was reported in \cite{wl}, where the initial value problem \eqref{nsB} was investigated  in the framework of modulation spaces. 

 Coming back to the good Boussinesq equation \eqref{good-bouss} in $\R^n$,   Geba and Witz in \cite{gebawitz2020} proved in the cases $n=2,3$ with the quadratic nonlinearity $\pm u^2$ that \eqref{good-bouss} is locally well-posed for $(u, u_t)|_{t=0}\in H^s(\R^n)\times H^{s-2}(\R^n)$ with $s>-1/4$. Indeed,  by an argument inspired by an approach due to Kishimoto and Tsugawa \cite{kist}, they  reformulated the associated initial value problem as the Cauchy problem for a nonlinear Schr\"{o}dinger equation with a nonlocal nonlinearity and initial data in $H^s(\R^n)$, and proved the associated bilinear estimates in a Bourgain space by using Tao's methodology in \cite{tao}. Relying on the Strichartz-type estimates for
the Schr\"{o}dinger group, Farah  in  \cite{farah} rewrited the good Boussinesq equation in the form of
\[
u_{tt}+\Delta^2u=\Delta(u+f(u))
\] 
and showed the local (unconditional) well-posedness with data in $(u, u_t)|_{t=0}\in H^s(\R^n)\times H^{s-2}(\R^n)$, $s\geq0$, under some assumptions on the nonlinearity. We should note that this approach is not appropriate for \eqref{sBeq} due to the presence of triharmonic operator.

To study \eqref{nsB}, we are going to derive dispersive estimates which enable us to find lower Sobolev index of well-posedness. As usual, these estimates are tightly related to the phase function  of \eqref{nsB}, that is, $$m(\xi)=\sqrt{|\xi|^2+\beta|\xi|^4+|\xi|^6}.$$ 
Contrary to the one-dimensional case, it is not easy to derive uniform estimates directly. In the case $n=1$ with $\beta>-2$,  the Van der Corput lemma was applied in \cite[Lemma 4.3]{esf-leva} to derive the following uniform estimates  

	\begin{equation}\label{osci-int}
	\sup_{x\in\R}\left|\int_{\R} e^{i (t m(\xi) +x\xi ) }  d  \xi\right|\lesssim
\begin{cases}
			|t|^{-\frac13},&\beta\neq0 \\
			|t|^{-\frac13}+|t|^{-\frac12},&\beta=0 
		\end{cases}
	\end{equation}
 for all $t\neq0$. In the case $n\geq2$, we take advantage of the symbol $m$ being radial to reduce the problem  to an oscillatory integral in  one dimension by using the Bessel function. Such a situation was already investigated in several papers. Among them, we refer to techniques developed by Cho and Ozawa \cite{choozawa} and Guo, Peng, and Wang \cite{gpw}, where  a general radial (nonhomogeneous) phase function was considered (see also \cite{cholee}). But one of the principal conditions in these articles is that the radial derivative of the associated phase function should not be  sign-changing. This condition fails to be valid for the phase function of \eqref{nsB} when $\beta=-1$. We overcome this difficulty via frequency localization by separating the high and low frequencies, and find the decay rate $t^{-\frac n2}$ for the localized $n$-dimensional version of \eqref{osci-int} (see Lemma \ref{lm-dispest}). Next, we  can suitably derive   decay estimates, by using the dyadic decomposition and  decay properties of the Bessel function. Since our results hold in the one-dimensional case, it is worth noting that our dispersive estimates improve locally ones obtained in \cite{esf-leva}.

\vspace{2mm}
Our first result is as follows.
\begin{theorem}\label{local-theorem-quad}
	Let $n\ge 2$, $\beta=\pm1$  and $f(u)=\pm u^2$. The initial value problem \eqref{nsB} is locally well-posed in  $H^s(\R^n)\times H^{s-3}(\R^n)$ for 
	$$
s >  \begin{cases}
 (n-3)/4, \qquad  & for \quad 2\le n \le 6,
\\
 (n-5)/2, \qquad   & for \quad  n \ge 7.
\end{cases}
$$
\end{theorem}

\begin{remark}
	Following the proof of Theorem \ref{local-theorem-quad} in the one-dimensional case (see Lemma \ref{lm-bilest}), one can observe that this theorem is also valid for $n=1$ with $s\geq-1/8$. The  gap between the Sobolev index of well-posedness in  Theorem \ref{local-theorem-quad} and result of \cite{wangesfahani}  suggests that our results may be improved in the case $n\geq2$. We will treat this in the future.
\end{remark}

We also derive appropriate localized Strichartz estimate to extend Theorem \ref{local-theorem-quad} to the cubic case. The following theorem improve the results obtained in \cite{efw}.
\begin{theorem}\label{local-theorem-cubic}
	Let  $n\ge 1$, $\beta=\pm1$  and $f(u)=\pm u^3$. The initial value problem \eqref{nsB} is locally well-posed in  $H^s(\R^n)\times H^{s-3}(\R^n)$ for
	$$
s >  \begin{cases}
 -1/6, \qquad  & for \quad  n=1,
\\
 (2n-3)/6, \qquad  & for \quad 2\le n \le 4,
\\
 (2n-5)/4, \qquad   & for \quad  n \ge 5.
\end{cases}
$$
	\end{theorem}

\begin{remark}\label{global-remark}
One can note that the local solutions obtained from \eqref{local-theorem-cubic} can be extended globally in time. More precisely,   by using the energy conservation of \eqref{nsB} with $f(u)=\pm u^3$,
\[
	E(u(t))=\frac12\int_{R^n}|(-\Delta)^{-\frac12}u_t|^2+|u|^2+\beta|\nabla u|^2+ |\Delta u|^2 d x\pm\frac14 \int_{R^n} u^4d x=E(u(0)),
\]
one can show the global well-posedness of \eqref{nsB} with the nonlinearity $f(u)=u^3$ in the energy space  $H^2(\R^n)\times \dot H^{-1}(\R^n)$    if    $n\leq 6$. This can be also extended to $f(u)=-u^3$ under suitable conditions on the size of initial data (see e.g. \cite{efw}).
\end{remark}

\begin{comment}
\begin{remark}
	Notice that  the range of $\beta$ in \eqref{nsB}  can be extended to $\beta> \beta_0>-2$ (see \cite{esf-leva}). The proofs of Theorems \ref{local-theorem-quad} and  \ref{local-theorem-cubic} are also valid for $\beta>\beta_0>-2$ with natural modifications. One should consider also the fact that there exists $\beta_0$, close to $-2$, such that the radial derivative of the phase function $m$ for any $\beta\leq\beta_0$ vanishes at some $r_\beta=|x_\beta|>0$.
\end{remark} 
\end{comment}
Our second aim in this paper is to study the persistence of spatial analyticity for the solutions of \eqref{nsB}. There are a lot of papers  looking at spatial analyticity. We refer the readers to \cite{bgk,gk,hayashi,selberg-0,selberg}  and references therein.	
Since   the Paley-Wiener Theorem characterizes the radius of analyticity
of a function by decay properties of its Fourier transform, so it is natural to take the initial data in the Gevrey space $G^{\sigma,s}(\R^n)$. This space is introduced by the norm
\[
\|f\|_{G^{\sigma,s}(\R^n)}= \norm{e^{\sigma|\xi|}\angles{\xi}^s \hat{f}(\xi)} _{L^2(\R^n)},
\]
where $\hat{f}$ is the Fourier transform of $f$ and $\angles{\cdot}=1+|\cdot|$. It is clear that $G^{0,s}(\R^n)=H^s(\R^n)$. A simple modification of the Paley-Wiener Theorem (see \cite{katz}) shows that any function $f\in G^{\sigma,s}(\R^n)$ with $\sigma>0$  can be   holomorphically  extended to  the  symmetric complex  strip
\[
S_\sigma=\left\{x+iy\in\C^n; \;|y|<\sigma\right\}.
\]
In this context,  the parameter $\sigma>0$  is referred as the uniform radius of analyticity of $f$. These spaces   were introduced by Foias and Temam \cite{foiastemam} to study the spatial
analyticity of solutions of the Navier-Stokes equations. Given an initial data  with a uniform radius of analyticity $\sigma_0$, it is interesting to know that whether the solution $u(t)$ with 
$t > 0$ also has a uniform radius of analyticity $\sigma = \sigma(t) > 0$. In the one-dimensional case of \eqref{nsB} with the quadratic nonlinearity, this problem was investigated in \cite{bfh} by using bilinear estimates in Bourgain type spaces associated with the Gevrey space. We, here, follow the method used in \cite{sel-tes}  in the context of the 1D Dirac-Klein-Gordon equations to find   a    lower bound on $\sigma(t)$ for the cubic nonlinearity $f(u)=u^3$. The argument is based on an approximate conservation law of \eqref{nsB}. We remark that    the energy 
$
E(u(t))%=\frac12\int_{R^n}|(-\Delta)^{-\frac12}u_t|^2+|u|^2+\beta|\nabla u|^2+ |\Delta u|^2 d x+\frac14\int_{R^n} u^4d x,
$
 is formally conserved by the flow of \eqref{nsB}. Using the ideas of \cite{dmt}, we introduce a modified Gevrey norm
\[
\|f\|_{H^{\sigma,s}(\R^n)}= \norm{\cosh(\sigma|\xi|)\angles{\xi}^s\hat{f}(\xi)} _{L^2(\R^n)}.
\]
More precisely, it is seen from the fact $\cosh(\sigma|\xi|)\sim e^{\sigma|\xi|}$,  the norms $\|\cdot\|_{H^{\sigma,s}(\R^n)}$ and $\|\cdot\|_{G^{\sigma,s}(\R^n)}$ are equivalent. The idea of definition of $H^{\sigma,s}(\R^n)$ is delicately connected to the decay rate of  exponential weight of $G^{\sigma,s}(\R^n)$-norm. In fact, it can be technically seen that the desired decay rate-in-time of the radius of analyticity $\sigma $ is obtained from the algebraic estimate
\[
e^{\sigma|\xi|}-1\leq(\sigma|\xi|)^\ell e^{\sigma|\xi|},\qquad\ell\in[0,1],
\]
 which could provide a decay rate of order $t^{-1/\ell}$ for $\ell\in(0,1]$. In the new space $H^{\sigma,s}(\R^n)$, if nonlinear estimates  of the approximate  conservation law can dissolve the weight $|\xi|^{2\ell}$, then the estimate
 \[
 \cosh(\sigma|\xi|)-1\leq(\sigma|\xi|)^{2\ell} \cosh(\sigma|\xi|),\qquad\ell\in[0,1] 
 \]
can be used to provide a decay rate of order $t^{-1/(2\ell)}$ for $\ell\in(0,1]$.

Now we state our second result.

\begin{theorem}\label{theo-1-anal}
	Let $n=1$ and $\sigma_0>0$. Given the initial data $(u_0,u_1)\in H^{\sigma_0,2}(\R)\times H^{\sigma_0,0}(\R)$, then the solution $u(t)$ of \eqref{nsB} for any $T>0$ satisfies
	\[
	u\in C([0,T];H^{\sigma_0,2}(\R))\cap C^1([0,T];H^{\sigma_0,0}(\R)),
	\]
	where $\sigma=\sigma(T)=\min \sett{\sigma_0,C T^{-\frac12}}$ with $C=C\left(\norm{u_0}_{H^{\sigma_0,2}(\R)}, \norm{u_1}_{H^{\sigma_0,0}(\R)}\right)$.
\end{theorem}

We notice from Remark \ref{global-remark} that the global well-posedness of \eqref{nsB} in $H^{\sigma_0,2}(\R)\times H^{\sigma_0,0}(\R)$ is extracted  via  $H^{\sigma,2}(\R)\hookrightarrow H^{\sigma}(\R)$, $\sigma\geq0$.

In the rest of paper, we first find the decay estimates of the phase function of \eqref{nsB}, then some dispersive estimates and the corresponding Strichartz estimates are derived in Section \ref{section-decay}. The binlinear and thrilinear estimates are proved in Sections \ref{sec-pfQ} and \ref{sec-pfC}, respectively. The last section is devoted to the proof of Theorem \ref{theo-1-anal}.

%%%%%%%%%%%%%%%%%%%%%%%%%%%%%%%%%%%%%%%%%%%%%%%%%%%%%%%%%%%%%%%%%%
\section{Decay Estimate}\label{section-decay}
In this section, we derive some estimates on the derivatives of   phase function of \eqref{nsB},
$$
m (r)=\sqrt{ r^2 +\beta r^4 + r^6 }  \qquad \qquad  ( r=|\xi| \ge 0 ).
$$
These estimates help us to find appropriate estimate of the associated oscillatory integral. 

First we observe some elementary properties of $m$.

\begin{lemma}\label{lm-mest}
Suppose that $m$ is given as above.
\begin{enumerate}[(i)]
	\item  There holds for any $r\geq0$ that
	\begin{equation}\label{m-est}
		m(r)  \sim r \angles{ r}^2.
	\end{equation}
\item If  $\beta=1$, then
\begin{align*}
	m'(r)\sim  \langle r\rangle ^2,\qquad 
	r^3 \langle r\rangle ^{-2}\lesssim m''(r)\lesssim r.
\end{align*}
\item If  $\beta=-1$ and $r\geq1$, then
\begin{align}\label{m1-est} 
m'(r)  \sim r ^2, 
\qquad
   m''(r) &\sim  r,
\end{align}
and
\begin{align}\label{mj-est} 
\left|m^{(j)}(r)\right|\underset{j}  \lesssim r^{3- j}  \qquad ( j \ge 3 ).
\end{align}
\end{enumerate}
\end{lemma}

\begin{proof}
These estimates follow from
\begin{align*}
m'(r) &= \frac{r \left[ 1+ 2 \beta r^2 + 3r^4\right]}{m(r)},
\\
m''(r) &=  \frac{ r \left(3\beta +(2\beta^2+10)r^2+9\beta r^4+6r^6\right)}{   (1	+\beta r^2 +r^4)^\frac32 }.
\end{align*}
\end{proof}

We now recall the Van der Corput Lemma (see \cite{stein}).

\begin{lemma}\label{lm-corput}
 Assume 
$g \in C^1(a, b)$, $\psi\in  C^2(a, b)$ and $|\psi''(r)|  \ge  A$ for all $r\in (a, b)$. Then 
\begin{align}
\label{corput} 
\Bigabs{\int_a^b e^{i t  \psi(r)}   g(r) \, dr}& \le C  (At)^{-1/2}  \left[ |g(b)| + \int_a^b |g'(r)| \, dr \right] ,
\end{align}
for some constant $C>0$ that is independent of $a$, $b$ and $t$.
\end{lemma}

Lemma \ref{lm-corput} holds even if $\psi'(r)=0$ for some $r\in (a, b)$.
However, if $|\psi'(r)|> 0$ for all $r\in (a, b)$, one can   obtain the following lemma by using an integration by parts. 
\begin{lemma}[\cite{ddt}]\label{lm-corput1}
Suppose that $g \in C^\infty_0(a,b)$ and $\psi\in C^\infty (a, b)$ with $|\psi'(r)|> 0$ for all $r\in (a, b)$. If 
  \begin{equation}\label{dervbd}
\max_{a\le r\le b}|\partial_r^j g(r)| \leq A, \qquad  \max_{a\le r\le b}  \Bigabs{\partial_r^j\left( \frac 1{\psi'(r)} \right)} \leq B
  \end{equation}
  for all $0 \le j \le \ N\in \N_0$,
  then 
\begin{align}
\label{corput'} 
\Bigabs{\int_a^b e^{i t  \psi(r)}   g(r) \, dr}& \lesssim A B^N  |t|^{-N} .
\end{align}
\end{lemma}

\vspace{5mm}
\noindent \textbf{Notations}. 
We fix an even smooth function $\chi \in C_0^{\infty}(\mathbb R)$ such that
\begin{equation*}
 0 \le \chi \le 1, \quad
\chi_{|_{[-1,1]}}=1 \quad \mbox{and} \quad  \mbox{supp}(\chi)
\subset [-2,2]
\end{equation*}
 and set
$$
\rho(s)
=\chi\left(s\right)-\chi \left(2s\right).
 $$
 For a dyadic number
  $\lambda \in  2^\Z$,  we set $\rho_{\lambda}(s):=\rho\left(s/\lambda\right)$, and thus $\supp \rho_\lambda= 
\{ s\in \R: \lambda/ 2 \le |s| \le 2\lambda \}$. 
Now define the frequency projection $P_\lambda$ via
\begin{align*}
\widehat{P_{\lambda} f}(\xi)  = \rho_\lambda(|\xi|)\widehat { f}(\xi) .
 \end{align*}
We sometimes write $f_\lambda:=P_\lambda f $, so that
\[ f=\sum_{\lambda  } f_\lambda ,\]
where summations throughout the paper are done over dyadic numbers in $ 2^\Z$.

\subsection{Dispersive estimates} 
In this section, we derive some localized dispersive estimate which are main tools in obtaining Strichartz estimates. To do so,  we will find 
a  frequency localized  $L^1_x-L^{\infty}_x$ decay estimate for the linear propagator associated to \eqref{nsB}, viz.
$$
\mathcal S_{m}(  t) := e^{ it m(D)},
$$
where the symbol $D$   presents the operator $i\nabla$.

\begin{lemma}[Localized dispersive estimate] \label{lm-dispest}
Let $n\ge 1$ and $\lambda\gg1  $.
Then
\begin{align}
\label{dispest}
\| \mathcal S_{m}( t)P_{\lambda} f  \|_{L^\infty_x(\R^n)} &\lesssim    ( \lambda |t|)^{-\frac n2}  \| f\|_{L_x^1(\R^n)} \end{align}
for all $f \in \mathcal{S}(\R^n)$.
 \end{lemma}
We can write
\[
\left[ \mathcal S_{m}(t) f_\lambda \right](x) 
= (I_{\lambda} (\cdot, t)\ast f)(x),
\]
where
\begin{equation}\label{Idef}
 I_{\lambda} (x, t)
=\lambda^n \int_{\R^n} e^{i \lambda x \cdot \xi+ it {m}(\lambda \xi)}  \rho(|\xi|) \, d\xi .
\end{equation}
By Young's inequality
\begin{equation}\label{younginq}
\| S_{m}(t) f_\lambda  \|_{L^\infty_x(\R^n)} \le \| I_{\lambda (\cdot, t)} \|_{L^\infty_x(\R^n)}  \|f\|_{L_x^1(\R^n)},
\end{equation}
and therefore \eqref{dispest} reduces to proving
\begin{equation}\label{diseEst-reduc}
\|I_{\lambda} (\cdot, t) \|_{L^\infty_x(\R^n)}  \lesssim      ( \lambda |t|)^{-\frac n2} .
\end{equation}

We provide the proof only for $n\ge 2$ and $\beta=-1$ as the one dimensional case or the case $\beta=1$ are easier to treat.

\subsection{Proof of \eqref{diseEst-reduc}}
We assume without loss of generality that $t>0$.
Using polar coordinates, we can write

\begin{equation}\label{I-eq}
I_{\lambda} (x, t) =\lambda^n  \int_{1/2}^2  e^{it m(\lambda r) }  (\lambda r|x|)^{-\frac{n-2}{2}}  J_{\frac{n-2}{2}}( \lambda r |x|)   r^{n-1}  \rho(r) \, dr,
\end{equation}
where $J_k(r)$ is the Bessel function:
$$
J_k(r)=\frac{ (r/2)^k}{(k+1/2) \sqrt{\pi}} \int_{-1}^1  e^{ir s} \left(1-s^2\right)^{k-1/2} \, ds \quad \text{for} \ k>-1/2.
$$
  The Bessel function $J_k(r)$ satisfies the following properties for $k>-1/2$ and $r>0$,
\begin{align}
\label{Jm1}
J_k (r) &\le Cr^{k} ,
\\
\label{Jm2}
J_k(r)& \le C r^{-1/2} ,
\\
\label{Jm3}
\partial_r \left[ r^{-k} J_k(r)\right] &= -r^{-k} J_{k+1}(r)
\end{align}
Moreover, we can write
\begin{equation}
\label{J0est}
 r^{- \frac{n-2}2 }J_{ \frac{n-2}2}(s)= e^{is} h(s)  +e^{-is}\bar h(s)
\end{equation}
for some function $h$ satisfying the decay estimate 
\begin{equation}
\label{h-est}
| \partial_r ^j h(r)|\le C_j \angles{r}^{-\frac{n-1}2-j}  \quad \text{for all} \ j\ge 0. 
\end{equation}

We use the short hand
$$ 
 m_{\lambda}(r) = m(\lambda r),  \qquad \tilde J_a(r)= r^{-a} J_a(r), \qquad \tilde\rho(r)=r^{n-1} \rho(r).$$

Hence,
\begin{equation}\label{I-eqq}
I_{\lambda} (x, t) = \lambda^n  \int_{1/2}^2  e^{it  m_{\lambda}(r)}  \tilde J_{\frac{n-2}{2}}( \lambda r |x|)  \tilde  \rho(r) \, dr.
\end{equation}

\vspace{2mm}

We treat the cases $  |x|\lesssim \lambda^{-1}$ and $  |x|\gg  \lambda^{-1}$ separately. 
\subsubsection{Case 1: $  |x|\lesssim  \lambda^{-1}$}
By \eqref{Jm1} and \eqref{Jm3} we have for all $ r\in (1/2, 2)$ the estimate

\begin{equation}
\label{J0derv-est}
\left| \partial_r ^j \left[  \tilde J_{ \frac{n-2}2 }( \lambda r |x|)  \tilde\rho(r) \right]\right| \underset{j}  \lesssim 1  \qquad  ( j \ge 0).
\end{equation}

It is straightforward to see from  Lemma \ref{lm-mest} that
\begin{equation}\label{mlamb-invest}
\max_{ 1/2 \le r \le 1 }\Bigabs {\partial_r ^j \left(  \frac 1{  m'_{ \lambda}(r)  } \right)} \underset{j}  \lesssim \lambda^{- 3}  \qquad   (j \ge 0).
\end{equation}

Applying Lemma \ref{lm-corput1} with \eqref{J0derv-est}-\eqref{mlamb-invest} and $N=  n/2$
to \eqref{I-eqq}, we obtain
\begin{equation}\label{Iest-2}
\begin{split}
| I_{\lambda} (x, t) |
&\lesssim \lambda^n \cdot \lambda^{- \frac {3n}2} t^{-   \frac n2}  \lesssim  (\lambda t)^{-   \frac n2} .
\end{split}
\end{equation}

\subsubsection{Case 2: $  |x|\gg  \lambda^{-1}$ }
Using \eqref{J0est} in \eqref{I-eq} we write
\begin{align*}
 I_{\lambda} (x, t) 
 &=\lambda^n \left\{\int_{1/2}^2  e^{it \phi^+_{\lambda} (r)  }  h(\lambda r |x|)  \tilde \rho(r) \, dr +  \int_{1/2}^2  e^{-it \phi^-_{\lambda} (r)  } \bar  h(\lambda r |x|)  \tilde\rho(r) \, dr \right\},
\end{align*}
where 
$$
\phi^\pm_{\lambda} (r)=    \lambda r|x|/t  \pm  m (_\lambda ( r) .
$$
Set $H_{\lambda}( |x|, r) =h(\lambda r |x|)  \tilde \rho(r)$. In view of \eqref{h-est} we have 
\begin{equation}
\label{Hest}
 \max_{1/2 \le r\le 2}\Bigabs {\partial_r ^j H_{\lambda}( |x|, r)  }    \lesssim   (\lambda |x|)^{-\frac{n-1}2}  \qquad  ( j \ge 0),
\end{equation}
 where we also used the fact $\lambda |x|\gg 1$

Now 
we write 
$$
 I_{\lambda} (x, t) 
= I^+_{\lambda} (x, t) 
+  I^-_{\lambda} (x, t) ,
$$
where
\begin{align*}
 I^+_{\lambda} (x, t) 
 &= \lambda^n \int_{1/2}^2  e^{it \phi^+_{\lambda} (r)  } H_{\lambda}( |x|, r)  \, dr ,
 \\
I^-_{\lambda} (x, t) &= \lambda^n
   \int_{1/2}^2  e^{-it \phi^-_{\lambda} (r)  }\bar H_{\lambda}( |x|, r)  \, dr .
\end{align*}

Observe that
$$
\partial_r \phi^\pm_{\lambda} (r)=  \lambda  \left[ |x|/t \pm  m_\lambda'( r) \right],\qquad \partial_r^2\phi^\pm_{\lambda} (r)=     \pm  \lambda^2 m_\lambda''( r),
$$
and hence by Lemma \ref{lm-mest}, 
\begin{equation}
\label{phi'+:est}
|\partial_r \phi^+_{\lambda} (r)|\gtrsim  \lambda^3
\qquad 
|\partial^2_r \phi^\pm_{\lambda} (r)| \sim   \lambda^3
\end{equation}
for all $ r\in (1/2, 2)$, where we also used the fact that $\lambda \gg 1$ and $m'$ is positive.

\subsubsection*{\underline{Estimate for  $I^+_{\lambda} (x, t)$ } }

Analogous to estimate \eqref{mlamb-invest}, we notice that
  \begin{equation}
\label{Est-phi+'-1}
\max_{1/2 \le r \le 2}\Bigabs {\partial_r ^j \left( \left[ \partial_r \phi_\lambda^+ (r) \right]^{-1} \right)}  \underset{j} \lesssim \lambda^{- 3}\qquad   (j \ge 0).
\end{equation} 
Applying Lemma \ref{lm-corput1} with \eqref{Hest}, \eqref{Est-phi+'-1}  and $N=  n/2$
to $I^+_{\lambda} (x, t) $ we obtain
\begin{equation}\label{Iest-3}
\begin{split}
| I^+_{\lambda} (x, t) |
&
\lesssim  \lambda^n  \cdot  (\lambda |x|)^{-\frac{n-1}2}  \cdot \lambda^{- \frac {3n}2}  t^{- \frac n2} 
\lesssim   (\lambda t)^{- \frac n2} ,
\end{split}
\end{equation}
where we also used the fact that  $\lambda |x|\gg 1$.

\subsubsection*{\underline{Estimate for $I^-_{\lambda} (x, t)$}}
We treat the the non-stationary and stationary cases separately. In the non-stationary
case, where 
$$ |x | \ll   \lambda^{2} t \quad \text{or} \quad |x| \gg   \lambda  ^{2} t, $$ we have 
$$
|\partial_r \phi^-_{\lambda} (r)|\gtrsim   \lambda ^{3},
$$
and hence $I^-_{\lambda} (x, t)$ can be estimated in exactly the same way as $I^+_{\lambda} (x, t)$ above, and satisfies the same bound.

So it remains to treat the stationary case: 
$$
|x | \sim  \lambda ^{2} t.
$$
 In this case,  
we use Lemma \ref{lm-corput}, \eqref{phi'+:est} and   \eqref{Hest} to obtain 
\begin{equation}\label{Iest-station}
\begin{split}
| I^-_{\lambda} (x, t) 
&\lesssim \lambda^n \left( \lambda^3  t \right)^{-\frac12}\left[ |H_\lambda^- (x, 2) |+ \int_{1/2}^2 | \partial_r  H_\lambda^- (x, r)| \, dr\right]
\\
&\lesssim   \lambda^{n-\frac 32 }  t^{-\frac12} \cdot (\lambda |x|)^{-\frac {n-1}2}
\\
& \lesssim     ( \lambda t)^{-\frac n2}.
\end{split}
\end{equation}
where we also used the fact that $H_\lambda^- (x, 2) =0$.

\subsection{Strichartz estimates}
Having the localized dispersive estimate in our hand, we are ready to derive the localized Strichartz estimates.

\begin{lemma}[Localized Strichartz estimates]\label{lm-LocStr}
 Let $\lambda\gg 1$, and
 assume that the pair $(q, r)$ satisfies
  \begin{equation} \label{admissible}
 q> 2, \ r\ge 2 \quad \text{and} \quad\frac2q +  \frac nr=\frac n2 \, .
 \end{equation} 
 Then
 \begin{align}
\label{Strest1d}
\norm{ \mathcal S_{m}( t) f_{\lambda}}_{ L^{q}_{t} L^{r}_{ x} (\R^{n+1}) } \lesssim \lambda ^{-\frac 1 q}  
\norm{  f_{\lambda}}_{ L^2_{ x}(\R^n )} ,
\end{align}
for all  $f \in \mathcal{S}(\R^n)$.
Moreover, if $b>1/2$, we have
\begin{equation}
\label{Str-transfer}
\norm{ u_\lambda }_{L^q_{t} L^r_{x}  (\R^{n+1}) } \lesssim 
\lambda^{-\frac 1q}
\norm{u_\lambda}_{  X^{0, b} },
\end{equation}
where \begin{align*}
\| u \|_{X^{s,b}} := \left\| \angles{ \xi }^{s} \angles{ |\tau| - m(\xi)}^b \widetilde{u} (\xi, \tau) \right\|_{L^2_{\tau, \xi}}.
\end{align*}
\end{lemma}
 \begin{proof}
  We shall use the Hardy-Littlewood-Sobolev inequality which
asserts that
\begin{equation}
\label{HLSineq}
\norm{|\cdot |^{-\gamma}\ast f}_{L^a(\R)} \lesssim \ \norm{ f}_{L^b(\R)} 
\end{equation}
whenever $1 < b < a < \infty$ and $0 < \gamma< 1$ obey the scaling condition
$$
\frac1b=\frac1a +1-\gamma.
$$
First note that \eqref{Strest1d} holds true for the pair $(q, r)=(\infty, 2)$ 
as this is just the energy inequality.  So we may assume $q\in(2, \infty)$.

Let $q'$ and $r'$ be the conjugates of $q$ and $r$, respectively, i.e., $q'=\frac q{q-1}$ and  $r'=\frac r{r-1}$.
 By the standard $TT^*$--argument, \eqref{Strest1d} is equivalent to the estimate 
\begin{equation}
\label{TTstar}
\norm{ TT^\ast F }_{L^{q}_{t} L^{r}_{ x} (\R^{n+1}) } \lesssim \ \lambda ^{-\frac 2q}  
\norm{ F  }_{ L^{q'}_{ t} L_x^{r'}(\R^{n+1} )},
\end{equation}
where 
\begin{equation}\label{TTastF}
\begin{split}
 TT^\ast F (x, t)&= \int_{\R^n}  \int_\R e^{i  x  \xi+  i(t-s) {m}(  \xi)}  \rho^2_\lambda (\xi)   \widehat{F}( \xi, s)\, ds  d\xi
 \\
 &= \int_\R  K_{\lambda,  t-s} \ast F( \cdot,  s) \, ds,
 \end{split}
\end{equation}
with
\begin{align*}
 K_{\lambda ,t}(x)&= \int_{\R^n}  e^{i  x  \xi+ it  {m}(  \xi)}  \rho^2_\lambda (\xi)   \, d\xi.
\end{align*}
Since $$K_{\lambda, t} \ast g (x)= \mathcal S_{m}(t)  P_\lambda g_\lambda (x)$$  
 it follows from \eqref{dispest} that
\begin{equation}\label{kest1}
\|K_{\lambda, t} \ast g \|_{L_x^{\infty}(\R^n)} \lesssim  (\lambda  |t|)^{-\frac n2}  \|g\|_{L_x^{1}(\R^n)}.
\end{equation}
On the other hand, we have by Plancherel 
\begin{equation}\label{kest2}
\|K_{\lambda, t} \ast g \|_{L_x^{2}(\R^n)} \lesssim    \|g\|_{L_x^{2}(\R^n)}.
\end{equation}
So interpolation between \eqref{kest1} and \eqref{kest2} yields
\begin{equation}\label{kest3}
\|K_{\lambda, t} \ast g \|_{L_x^{r}(\R^n)} \lesssim   (\lambda  |t|)^{- \frac n2\left(1-\frac 2r\right)} \|g\|_{L_x^{r'}(\R^n)}
\end{equation}
 for all $  r \in[2, \infty].$

Applying Minkowski's inequality to \eqref{TTastF}, and then  \eqref{kest3} and  \eqref{HLSineq}
with $(a, b)=(q , q' )$ and $\gamma= n/2-n/r=2/q$,
 we obtain
 \begin{align*}
\norm{TT^\ast F }_{L^{q}_{t} L^{r}_{ x} (\R^{n+1})}
&\le \norm{   \int_\R \norm{ K_{\lambda, t-s,} \ast
   F(s, \cdot) }_{L_x^r (\R^n)}  \, ds}_{L^{q}_t(\R)}
  \\
 &\lesssim  \lambda ^{ -\frac2q}   \norm{  \int_\R  |t-s|^{-\frac2q }
  \norm{ F(s, \cdot) }_{ L_x^{r'}(\R^2)}  \, ds }_{L_t^{q}(\R)}
   \\
 &\lesssim \lambda ^{ -\frac2q}   \norm{  
  \norm{ F }_{L_x^{r'}(\R^n) }  }_{L^{q'}_{ t} (\R)}
    \\
 &=  \lambda ^{ -\frac2q} 
  \norm{ F  }_{ L^{q'}_{ t} L_x^{r'}(\R^{n+1})} \, ,
\end{align*}
which is the desired estimate \eqref{TTstar}.
 \end{proof}

\section{Local solution }\label{Q-bilinear-section}
In this section, we prove a bilinear and trilinear estimate which is main ingredient of the proof of Theorem \ref{local-theorem-quad}. We give the proof only for $\beta=-1$ as the case $\beta=1$  easier to treat. With $\beta=-1$ the phase function becomes
$$
 m(\xi)=|\xi|\sqrt{1-|\xi|^2+|\xi|^4}.
$$

By Duhamel's principle, the solutions of \eqref{nsB} are rewritten as the integral equation
\begin{equation}\label{contract-op}
	u(t)=  \partial_t \mathcal W_m(t)u_0 + \mathcal   W_m(t)u_1
	+\int_0^t \mathcal  W_m(t-t')  \Delta( f(u)(t')) \; dt',
\end{equation}
where
\begin{align*}
\mathcal  W_m(t) &=\frac{\sin ( t  m(D) )} { m(D)}= \frac{e^{itm(D)} + e^{-itm(D)}  }{ 2im(D)} = \frac{ S_m(t) - S_m(-t)  }{ 2im(D)} .
\end{align*}

By the standard argument (see e.g. \cite{linares-ponce}), it is known that 
local well-posedness follows from
the following bilinear and trilinear estimates. So we   omit the proof of Theorem \ref{local-theorem-quad} (see \cite{esfahanifarah}).
\begin{lemma}
\label{lm-bilest}
Let $n\ge 2$, $1/2<b<1$ , $0 < T < 1$, and $s$ is as in Theorem \ref{local-theorem-quad}. Then
\begin{equation}
\label{biest1}
\norm{  |D| \angles{D}^{-2} (u_1 u_2) }_{X^{s, b-1}}  
\lesssim   T^{1-b}
\norm{u_1}_{ X^{s, b} } 
\norm{u_2}_{X^{s, b}} .
\end{equation}

\end{lemma}

\vspace{2mm}

\begin{lemma}
\label{lm-bilest-cubc}
Let $n\ge 1$, $1/2<b<1$ , $0 < T < 1$ and $s$ be the same as in Theorem \ref{local-theorem-cubic}. Then
\begin{equation}
\label{biest1-c}
\norm{  |D| \angles{D}^{-2} (u_1 u_2 u_3) }_{X^{s, b-1}}  
\lesssim  \prod_{j=1}^3
\norm{u_j}_{ X^{s, b} } .
\end{equation}

\end{lemma}

To prove Lemmas \ref{lm-bilest} and \ref{lm-bilest-cubc}, we need the localized bilinear and trilinear estimates in Lemmas \ref{L2lemma} and \ref{L2lemma-cub} below, which are consequences of the frequency localized Strichartz estimates in Lemma \ref{lm-LocStr}.

\vspace{2mm}
We recall the Bernstein inequality which is very helpful in our analysis.
\begin{lemma}[\cite{chemin}]
	For $1 \leq r \leq q \leq\infty$ and $k\geq0$,
	\[
	\|D^kP_\lambda f\|_{L^q(\R^n)}\lesssim \lambda^{k+n(\frac1r-\frac1q)}	\|P_\lambda f\|_{L^r(\R^n)}.
	\]
\end{lemma}

\begin{lemma}
\label{L2lemma}
Let $n\ge 1$, $1/2<b<1$ and  $0 < T < 1$.
Then
\[
\norm{ |D| \angles{D}^{-2} P_{\lambda_3} \left(P_{\lambda_1}  u_1 P_{\lambda_2} u_2 \right) }_{L_T^2 L^2_x}  
\lesssim   \lambda_3 \angles{\lambda_3}^{-2}   B(\lambda)
\norm{P_{\lambda_1} u_1 }_{ X^{0, b} } 
\norm{P_{\lambda_2} u_2  }_{X^{0, b}},
\]
where 
\begin{equation}\label{B-q1}
B(\lambda) \sim [\min (\lambda_1, \lambda_2)]^\frac n2.
\end{equation}
Moreover, if $\max (\lambda_1, \lambda_2)\gg 1$, then we can take
\begin{equation}\label{B-q2}
B(\lambda) \sim 
\begin{cases} 
  [\max (\lambda_1, \lambda_2)]^{-\frac14 } \ \qquad  &\text{if} \quad   n=1,
\\
[\min (\lambda_1, \lambda_2)]^{ \frac n2 -1 +2\delta }   [\max (\lambda_1, \lambda_2)]^{-\frac12 + \delta}  \qquad  &\text{if} \quad   n\ge 2.
\end{cases}
\end{equation}
for  sufficiently small $\delta >0$.
\end{lemma}

\begin{proof}
In view of the Bernstein   inequality, it suffices to prove
$$
 \norm{P_{\lambda_1} u_1 P_{\lambda_2} u_2 }_{L_T^2L^2_x} \lesssim B(\lambda)
\norm{P_{\lambda_1} u_1 }_{ X^{0, b} } 
\norm{P_{\lambda_2} u_2  }_{X^{0, b}},
$$
By symmetry, we may assume 
$ \lambda_1 \le \lambda_2 $. 
By the H\"{o}lder  and Sobolev inequalities,
\begin{align*}
  \norm{P_{\lambda_1} u_1 P_{\lambda_2} u_2}_{L_T^2L^2_x} &\le T^{\frac 12} 
  \norm{P_{\lambda_1} u_1 }_{ L_T^\infty L_x^\infty  } 
  \norm{ P_{\lambda_2} u_2}_{ L_T^\infty L_x^2  }
  \\
 & \lesssim  T^{\frac 12} \lambda_1^\frac n2
  \norm{P_{\lambda_1} u_1 }_{ L_T^\infty L_x^2  } 
 \norm{ P_{\lambda_2} u_2}_{ L_T^\infty L_x^2  }
  \\
 & \lesssim T^{\frac 12}\lambda_1^\frac n2
  \norm{P_{\lambda_1} u_1 }_{X^{0, b}}
  \norm{  P_{\lambda_2} u_2 }_{X^{0, b}}
\end{align*}
This proves \eqref{B-q1}.

\vspace{2mm}
Next, assume $\lambda_2 \gg 1$. If $n=1$, then by the H\"{o}lder inequality in $t$,  the Bernstein   inequality in $x$, and finally Lemma \ref{lm-LocStr}, we get
\begin{align*}
  \norm{P_{\lambda_1} u_1 P_{\lambda_2} u_2}_{L_T^2L^2_x} &\le  T^{\frac 14}   \norm{ P_{\lambda_1} u_1  }_{ L^\infty_T L^{2}_x  }
  \norm{P_{\lambda_2} u_2}_{ L_T^4 L_x^\infty } 
  \\
 & \lesssim  T^{\frac 14}   \lambda_2^{ -\frac 14 }
  \norm{P_{\lambda_1} u_1 }_{X^{0, b}}
  \norm{ P_{\lambda_2} u_2 }_{X^{0, b}}.
\end{align*}
 If $n\ge 2$, we choose the Strichartz admissible pair 
$$
q=\frac {2} {1-2\delta}, \qquad r= \frac{2qn}{qn-4} .
$$
Then applying the H\"{o}lder inequality in $t$,  the Bernstein   inequality in $x$, and finally Lemma \ref{lm-LocStr}, we get
\begin{align*}
  \norm{P_{\lambda_1} u_1 P_{\lambda_2} u_2}_{L_T^2L^2_x} &\le  T^{\frac 12 -\frac 1q}   \norm{ P_{\lambda_1} u_1 }_{ L^\infty_T L^{\frac{qn}2}_x  }
  \norm{ P_{\lambda_2} u_2}_{ L_T^q L_x^r } 
  \\
  &\le  T^{\frac 12 -\frac 1q}  \lambda_1^{ \frac n2 -\frac 2q}  \lambda_2^{-\frac1q} \norm{ P_{\lambda_1} u_1}_{ L^\infty_T L^2_x  }  \norm{ P_{\lambda_2} u_2 }_{X^{0, b}}
  \\
 & \lesssim  T^{\delta} \lambda_1^{ \frac n2 -1 +2\delta }   \lambda_2^{ -\frac 12 + \delta}
  \norm{P_{\lambda_1} u_1 }_{X^{0, b}}
  \norm{ P_{\lambda_2} u_2}_{X^{0, b}}.
\end{align*}

\end{proof}

\begin{lemma}
\label{L2lemma-cub}
Let $n\ge 1$, $1/2<b<1$ and  $0 < T < 1$.
Assume that $\lambda_{\text{min}}$,  $\lambda_{\text{med}}$ and  $\lambda_{\text{max}}$
denote the minimum, median and maximum of $\lambda_1$,  $\lambda_2$ and  $\lambda_3$, respectively.
Then
\[
\norm{ |D| \angles{D}^{-2} P_{\lambda_4} (P_{\lambda_1}  u_1 P_{\lambda_2} u_2  P_{\lambda_3} u_3 )}_{L_T^2 L^2_x}  
\lesssim 
   \lambda_4 \angles{\lambda_4}^{-2}  B(\lambda)   \prod_{j=1}^3
\norm{P_{\lambda_j} u_j}_{ X^{0, b} } 
\]
where 
\begin{equation}\label{B-c1}
B(\lambda) =    ( \lambda_{ \text{min }}  \lambda_{ \text{med }} )^\frac n2.
\end{equation}
Moreover,
\begin{itemize}
\item  If $n=1$, we can take
\begin{equation}
\label{B-c2}
B(\lambda)=  \begin{cases} 
 \lambda_{ \text{min }}^\frac12 \lambda_{ \text{max} }^{-\frac14}  \qquad  &\text{if} \ \    \lambda_{\text{max}} \gg 1,
 \\
 ( \lambda_{ \text{med }} \lambda_{ \text{max} })^{-\frac14}  \qquad  &\text{if} \ \quad \lambda_{\text{med}} \gg 1.
\end{cases}
\end{equation}

\item  If $n\ge 2$ and $\lambda_{\text{max}} \gg 1$, we can take
\begin{equation}
\label{B-c3}
B(\lambda)= 
     \lambda_{ \text{min} }^{ \frac n2 }  \lambda_{ \text{med} }^{ \frac n2 -1 +2\delta }   \lambda_{ \text{max} }^{-\frac12 + \delta} .
\end{equation}
\end{itemize}

\end{lemma}

\begin{proof}
In view of the Bernstein   inequality, it suffices to prove
$$
 \norm{P_{\lambda_1}  u_1 P_{\lambda_2} u_2  P_{\lambda_3} u_3 }_{L_T^2L^2_x} \lesssim B(\lambda)\prod_{j=1}^3
\norm{P_{\lambda_j} u_j}_{ X^{0, b} } .
$$
By symmetry, we may assume 
$ \lambda_1 \le \lambda_2 \le \lambda_3$.
By the H\"{o}lder   and Sobolev inequalities,
\begin{align*}
 \norm{P_{\lambda_1}  u_1 P_{\lambda_2} u_2  P_{\lambda_3} u_3 }_{L_T^2L^2_x}  &\le T^{\frac 12} 
  \norm{P_{\lambda_1}  u_1}_{ L_T^\infty L_x^\infty  }  \norm{P_{\lambda_2} u_2}_{ L_T^\infty L_x^\infty  } 
  \norm{ P_{\lambda_3} u_3 }_{ L_T^\infty L_x^2  }
  \\
 & \lesssim  T^{\frac 12} ( \lambda_1  \lambda_2)^\frac n2 \prod_{j=1}^3
  \norm{ P_{\lambda_j} u_j }_{ L_T^\infty L_x^2  }
  \\
 & \lesssim T^{\frac 12} ( \lambda_1  \lambda_2)^\frac n2
  \prod_{j=1}^3
  \norm{ P_{\lambda_j} u_j }_{ X^{0, b}  }.
\end{align*}

Next, assume $\lambda_3 \gg 1$. 

\textbf{Case: $n=1$:}   By the H\"{o}lder inequality in $t$,  the Bernstein   inequality in $x$, and finally Lemma \ref{lm-LocStr}, we get
\begin{align*}
  \norm{P_{\lambda_1} u_1 P_{\lambda_2} u_2  P_{\lambda_3} u_3}_{L_T^2L^2_x} &\le  T^{\frac 14}   \norm{ P_{\lambda_1} u_1  }_{ L^\infty_T L^{\infty}_x  } \norm{ P_{\lambda_2} u_2  }_{ L^\infty_T L^{2}_x  }
  \norm{P_{\lambda_3} u_3}_{ L_T^4 L_x^\infty } 
  \\
 & \lesssim  T^{\frac 14}   \lambda_1^{ \frac 12 } \lambda_3^{ -\frac 14 } \prod_{j=1}^3
  \norm{ P_{\lambda_j} u_j }_{ X^{0, b}  }.
\end{align*}
In the case $\lambda_2 \gg 1$,
\begin{align*}
  \norm{P_{\lambda_1} u_1 P_{\lambda_2} u_2  P_{\lambda_3} u_3}_{L_T^2L^2_x} &\le   \norm{ P_{\lambda_1} u_1  }_{ L^\infty_T L^{2}_x  } \norm{ P_{\lambda_2} u_2  }_{ L^4_T L^{\infty}_x  }
  \norm{P_{\lambda_3} u_3}_{ L_T^4 L_x^\infty } 
  \\
 & \lesssim   ( \lambda_2 \lambda_3)^{ -\frac 14 } \prod_{j=1}^3
  \norm{ P_{\lambda_j} u_j }_{ X^{0, b}  }.
\end{align*}

 \textbf{Case: $n\ge 2$:} We choose the Strichartz admissible pair 
$$
q=\frac {2} {1-2\delta}, \qquad r= \frac{2qn}{qn-4}   \qquad (n\ge 2).
$$
Applying the H\"{o}lder inequality in $t$,  the Bernstein   inequality in $x$, and finally Lemma \ref{lm-LocStr}, we get
\begin{align*}
  \norm{P_{\lambda_1}  u_1 P_{\lambda_2} u_2  P_{\lambda_3} u_3  }_{L_T^2L^2_x}  &\le  T^{\frac 12 -\frac 1q}   \norm{ P_{\lambda_1}  u_1}_{ L^\infty_T L^{\infty}_x  }
  \norm{P_{\lambda_2}  u_2}_{ L_T^\infty L_x^{\frac{qn}2} }  \norm{P_{\lambda_3}  u_3}_{ L_T^q L_x^r } 
  \\
  &\le  T^{\frac 12 -\frac 1q} \lambda_1^{ \frac n2 }   \lambda_2^{ \frac n2 -\frac 2q}  \lambda_3^{-\frac1q}   \norm{ P_{\lambda_1}  u_1}_{ L^\infty_T L^2_x  }
  \norm{P_{\lambda_2}  u_2}_{ L_T^\infty L_x^2 }  \norm{ P_{\lambda_3}  u_3 }_{X^{0, b}}
  \\
 & \lesssim  T^{\delta} \lambda_1^{ \frac n2 }   \lambda_2^{ \frac n2 -1+ 2\delta} \lambda_3^{ -\frac 12 + \delta}
  \prod_{j=1}^3
  \norm{ P_{\lambda_j} u_j }_{ X^{0, b}  }.
\end{align*}

\end{proof}

\section{Proof of Lemma \ref{lm-bilest}} \label{sec-pfQ}

We recall (see e.g., \cite{tao-book})
 \begin{align}
 \label{TFactor}
\norm{u}_{X^{s, b-1}(T)} &\le C
T^{1-b}\norm{u}_{ L_T^2 H_x^{s}},
\end{align}
where $C$ is independent on $T$.
So in view of \eqref{TFactor}, the estimate \eqref{biest1} reduces to proving
\begin{equation}\label{biest2}
\norm{     |D| \angles{D}^{-2} (u_1 u_2) }_{L_T^2 H_x^{s}}  
\lesssim
\norm{u_1}_{ X^{s, b} } 
\norm{u_2}_{X^{s, b}}.
\end{equation}

By duality \eqref{biest2} reduces further to
\begin{equation}\label{duality-biest11}
\left| \int_0^T \int_{\R^n}   |D|\angles{D}^{s-2}\left( \angles{D}^{-s} u_1 \cdot  \angles{D}^{-s} u_2 \right) u_3 \ dx dt \right| \lesssim
\norm{u_1}_{  X^{0, b}} \norm{u_2}_{  X^{0, b} } \norm{ u_3 }_{  L_T^2 L_x^2 }.
\end{equation}

Decomposing $u_j= \sum_{\lambda_j >0} P_{\lambda_j} u_j$,
we have
\begin{equation}\label{duality-biestdecomp}
\text{LHS \eqref{duality-biest11} } \lesssim \sum_{ \lambda_1, \lambda_2 , \lambda_3 } \left| \int_0^T \int_{\R^n}   |D| \angles{D}^{s-2}  P_{\lambda_3} \left( \angles{D}^{-{s}} P_{\lambda_1} u_1 \cdot  \angles{D}^{-{s}}  P_{\lambda_2} u_2 \right)
 P_{\lambda_3} u_3 \ dx dt \right|.
\end{equation}

Let
\begin{align*}
a_{\lambda_j}:&= \norm{P_{\lambda_j} u_j}_{  X^{0, b} }, 
\qquad
a_{\lambda_3}:=  \norm{P_{\lambda_3} u_3}_{L^2_{T,x}   }  \qquad (j=1, 2),
\end{align*}
Then
\begin{align*}
 \norm{u_j}_{  X^{0, b} } \sim \|(a_{\lambda_j})\|_{l^2_{\lambda_j}}, \qquad
 \norm{u_3}_{  L^2_{T,x} } \sim \|(a_{\lambda_3})\|_{l^2_{\lambda_3}}  \qquad   (j=1, 2).
\end{align*}

Therefore, estimate \eqref{duality-biest11} reduces to 
\begin{equation}\label{duality-biest11-1}
\text{RHS \eqref{duality-biestdecomp}}\lesssim \prod_{j=1}^3 \|(a_{\lambda_j})\|_{l^2_{\lambda_j}}.
\end{equation}

We prove \eqref{duality-biest11-1} as follows.
By the Cauchy-Schwarz inequality, Lemma \ref{L2lemma} and  the Bernstein inequality, we have
\begin{equation}\label{maindecomp}
\begin{split}
\text{RHS \eqref{duality-biestdecomp}}
&\lesssim
 \sum_{\lambda_1, \lambda_2 , \lambda_3 }
\norm{    |D|\angles{D}^{s-2} P_{\lambda_3} \left( \angles{D}^{-{s}} P_{\lambda_1} u_1 \cdot  \angles{D}^{-{s}}  P_{\lambda_2} u_2 \right)}_{L_{T,x}^2} \norm{P_{\lambda_3} u_3}_{ L_{T,x}^2  }
\\
&\lesssim
\underbrace{\sum_{
	 \lambda_1, \lambda_2 , \lambda_3	}    C(\lambda) 
a_{\lambda_1} a_{\lambda_2} a_{\lambda_3} }_{:= S}, 
\end{split}
\end{equation}
where 
$$
C(\lambda) = B(\lambda) \lambda_3 \angles{\lambda_3}^{s-2} 
\angles{\lambda_1}^{-s} \angles{\lambda_2}^{-s }  
$$
with $B(\lambda)$ as in \eqref{B-q1}-\eqref{B-q2}.

\vspace{2mm}
So \eqref{duality-biest11-1} reduces 
to proving 
\begin{equation}
\label{Sest}
S \lesssim \prod_{j=1}^3 \|(a_{\lambda_j})\|_{l^2_{\lambda_j}}.
\end{equation}

By symmetry of our argument, we may assume 
$ \lambda_1 \le \lambda_2 $. 

If $\lambda_2 \lesssim 1$, then $B(\lambda)\lesssim 1$, and hence
$$
C(\lambda) \lesssim  \lambda_3 \lambda_1 ^\frac n2  \lesssim  \lambda_3 (\lambda_1 \lambda_2)^\frac n4 .
$$
Then applying the Cauchy-Schwarz inequality in $\lambda_1, 
  \lambda_2$ and $\lambda_3$, we
   obtain
\begin{align*}
S &\lesssim \sum_{
	 \lambda_1, \lambda_2 , \lambda_3 \lesssim 1	}     \lambda_3 (\lambda_1 \lambda_2)^\frac n4
a_{\lambda_1} a_{\lambda_2} a_{\lambda_3} 
 \lesssim  \prod_{j=1}^3 \|(a_{\lambda_j})\|_{l^2_{\lambda_j}}.
\end{align*}
So, from now on we assume $\lambda_2 \gg 1$.  We consider two cases:
\begin{enumerate}[(i)]
\item  $\lambda_3 \ll \lambda_1 \sim\lambda_2 $,
\item $\lambda_1 \ll \lambda_2 \sim\lambda_3$.
\end{enumerate}

\subsection{\underline{(i): $\lambda_3 \ll \lambda_1 \sim\lambda_2$}}

\subsubsection{\underline{ $n=1$}:}
In this case, we have
$$
C(\lambda) \sim  \lambda_3 \angles{\lambda_3}^{s-2} 
\lambda_2^{-2s-1/4 }.
$$
\begin{itemize}
\item
If $\lambda_3 \lesssim 1$, then
\begin{align*}
S &\lesssim \sum_{
	 \lambda_3\lesssim 1 \ll \lambda_1 \sim\lambda_2	}     \lambda_3
\lambda_2^{-2s-1/4  }
a_{\lambda_1} a_{\lambda_2} a_{\lambda_3} 
 \lesssim  \prod_{j=1}^3 \|(a_{\lambda_j})\|_{l^2_{\lambda_j}}
\end{align*}
provided $s> -1/8$.

\item If $\lambda_3 \gg 1$, then 
\begin{align*}
S &\lesssim \sum_{
	 1\ll \lambda_3 \ll \lambda_1 \sim\lambda_2	}     \lambda_3^{s-1}
\lambda_2^{-2s-1/4  }
a_{\lambda_1} a_{\lambda_2} a_{\lambda_3} 
 \lesssim  \prod_{j=1}^3 \|(a_{\lambda_j})\|_{l^2_{\lambda_j}}
\end{align*}
provided $ -1/8<s<1$. In the case $s\ge 1$, we write $\lambda_3^{s-1}
\lambda_2^{-2s-1/4  }=(\lambda_3/\lambda_2)^{s-1}
\lambda_2^{-s-5/4  }$, which can be used to bound the sum.
\end{itemize}

\subsubsection{ \underline{$n\ge 2$}:}
In this case, we have
$$
C(\lambda) \sim  \lambda_3 \angles{\lambda_3}^{s-2} 
\lambda_2^{-2s +n/2 -3/2 + 3\delta }.
$$

\begin{itemize}

\item
First assume $\lambda_3 \lesssim 1$. Then
\begin{align*}
S  &\lesssim \sum_{
	 \lambda_3\lesssim 1 \ll \lambda_1 \sim\lambda_2	}     \lambda_3 
\lambda_2^{-2s +n/2 -3/2 + 3\delta }
a_{\lambda_1} a_{\lambda_2} a_{\lambda_3} 
 \lesssim  \prod_{j=1}^3 \|(a_{\lambda_j})\|_{l^2_{\lambda_j}}.
\end{align*}
provided $s> (n-3)/4 + (3/2) \delta$.

\item Next, assume $\lambda_3 \gg 1$. 

\begin{itemize}

\item
If $n\le 6$ and $ (n-3)/4 + (3/2) \delta<s< 1$, then 
\begin{align*}
S &\lesssim \sum_{
	 1\ll \lambda_3 \ll \lambda_1 \sim\lambda_2	}     \lambda_3^{s-1}
\lambda_2^{-2s +n/2 -3/2 + 3\delta }
a_{\lambda_1} a_{\lambda_2} a_{\lambda_3} 
 \lesssim  \prod_{j=1}^3 \|(a_{\lambda_j})\|_{l^2_{\lambda_j}}.
\end{align*}
\item If $n\le 6$ and $s\ge 1$, then 
\begin{align*}
S &\lesssim \sum_{
	 1\ll \lambda_3 \ll \lambda_1 \sim\lambda_2	}     \lambda_3^{-\delta}
\lambda_2^{-s +n/2 -5/2 + 4\delta }
a_{\lambda_1} a_{\lambda_2} a_{\lambda_3} 
 \lesssim  \prod_{j=1}^3 \|(a_{\lambda_j})\|_{l^2_{\lambda_j}}
\end{align*}
since $ -s + (n-5)/2+ 4\delta \le -1/2 + 4\delta$.
\item If  $n\ge 7$ and $ s > (n-5)/2+ 4\delta$ (which implies $s> 1$), then it is clear from the proceeding estimate that $\mathcal S$ is summable.

\end{itemize}
\end{itemize}

\vspace{2mm}

\subsection{\underline{(ii): $\lambda_1 \ll \lambda_2 \sim\lambda_3$}}

\subsubsection{\underline{$n=1$}:}

\begin{itemize}
\item If $\lambda_1 \lesssim 1$, then $B(\lambda)\sim \lambda_1^{1/2}$, and hence  
$
C(\lambda) \sim  \lambda_1^{1/2}
\lambda_2^{ -1 }.
$
Therefore,
\begin{align*}
S  &\lesssim \sum_{
	 \lambda_1 \lesssim 1 \ll \lambda_2 \sim\lambda_3	}    \lambda_1^{1/2}
\lambda_2^{ -1 }
a_{\lambda_1} a_{\lambda_2} a_{\lambda_3} 
 \lesssim  \prod_{j=1}^3 \|(a_{\lambda_j})\|_{l^2_{\lambda_j}}
\end{align*}

\item If $\lambda_1 \gg 1$, then
$B(\lambda)\sim \lambda_1^{-1/4}$,
and hence
$
C(\lambda) \sim  
\lambda_1^{ -s-5/4 }.
$
Consequently,
\begin{align*}
S &\lesssim \sum_{
	1\ll  \lambda_1 \ll \lambda_2 \sim\lambda_3	}     
\lambda_1^{ -s-5/4 }
a_{\lambda_1} a_{\lambda_2} a_{\lambda_3} 
 \lesssim  \prod_{j=1}^3 \|(a_{\lambda_j})\|_{l^2_{\lambda_j}}
\end{align*}
provided $s>-5/4$.
\end{itemize}

\subsubsection{\underline{ $n\ge 2$}:}

In this case, we have
$$
C(\lambda) \sim   \lambda_1^{n/2 -1+ 2\delta } \angles{\lambda_1}^{-s}
\lambda_2^{ -3/2 + 3\delta }.
$$

\begin{itemize}

\item
First assume $\lambda_1 \lesssim 1$. Then 
\begin{align*}
S &\lesssim \sum_{
	 \lambda_1\lesssim 1 \ll \lambda_2 \sim\lambda_3	}     \lambda_1^{n/2 -1+ 2\delta } 
\lambda_2^{ -3/2 + 3\delta }
a_{\lambda_1} a_{\lambda_2} a_{\lambda_3} 
 \lesssim  \prod_{j=1}^3 \|(a_{\lambda_j})\|_{l^2_{\lambda_j}}.
\end{align*}

\item Next, assume $\lambda_1 \gg 1$.  Then 
\begin{align*}
S &\lesssim \sum_{
	 1\ll \lambda_1 \ll \lambda_2 \sim\lambda_3	}     \lambda_1^{-s+ n/2-5/2+ 2\delta}   
a_{\lambda_1} a_{\lambda_2} a_{\lambda_3} 
 \lesssim  \prod_{j=1}^3 \|(a_{\lambda_j})\|_{l^2_{\lambda_j}}
\end{align*}
provided 
$
s>(n-5)/2+ 3\delta$.

\end{itemize}

\vspace{2mm}

\textbf{ \underline{Conclusion:}} From the cases (i) and (ii), we conclude that the estimate \eqref{Sest} holds provided
$$
s > \max \left(    n/4-3/4 , \  n/2-5/2\right)=  \begin{cases}
 n/4-3/4, \qquad  &for \quad 2\le n \le 6,
\\
 n/2-5/2, \qquad &for \quad  n \ge 7.
\end{cases}
$$
Similarly, for the case $n=1$, \eqref{Sest} holds provided
$$
s>-1/8.
$$

\section{Proof of Lemma \ref{lm-bilest-cubc} }\label{sec-pfC}

By duality \eqref{biest1-c} reduces further to (recall \eqref{TFactor})
\begin{equation}\label{duality-biest11-c}
\left| \int_0^T \int_{\R^n}   |D|\angles{D}^{s-2}\left( \angles{D}^{-s} u_1 \cdot  \angles{D}^{-s} u_2 \angles{D}^{-s}  u_3 \right)  u_4 \ dx dt \right| \lesssim  T^{1-b}  \prod_{j=1}^3
\norm{u_j}_{ X^{0, b} } \norm{u_4}_{  L^2_{T,x}}
\end{equation}

Decomposing $u_j= \sum_{\lambda_j >0} P_{\lambda_j} u_j$,
we have
\begin{equation}\label{duality-biestdecomp-c}
\begin{split}
&	\text{LHS \eqref{duality-biest11-c} }\\& \quad\lesssim \sum_{\lambda_1, \lambda_2, \lambda_3, \lambda_4  } \left| \int_0^T \int_{\R^n}   |D| \angles{D}^{s-2} P_{\lambda_4} \left( \angles{D}^{-{s}} P_{\lambda_1} u_1   \cdot \angles{D}^{-{s}}  P_{\lambda_2} u_2   \cdot   \angles{D}^{-s} P_{\lambda_3} u_3 \right) P_{\lambda_4} u_4\ dx dt \right|.
\end{split}
\end{equation}

Let
\begin{align*}
a_{\lambda_j}:&= \norm{P_{\lambda_j} u_j}_{  X^{0, b} }, 
\qquad
a_{\lambda_4}:=  \norm{P_{\lambda_4} u_4}_{L^2_{T,x}   }  \qquad (j=1, 2, 3).
\end{align*}
Then
\begin{align*}
 \norm{u_j}_{  X^{0, b} } \sim \|(a_{\lambda_j})\|_{l^2_{\lambda_j}}, \qquad
 \norm{u_4}_{  L^2_{T,x} } \sim \|(a_{\lambda_4})\|_{l^2_{\lambda_4}}  \qquad   (j=1, 2, 3).
\end{align*}

Now, by the Cauchy-Schwarz inequality, Lemma \ref{L2lemma-cub} and  the Bernstein   inequality,
\begin{equation}\label{maindecomp-c}
\begin{split}
\text{RHS \eqref{duality-biestdecomp-c}}
&\lesssim
 \sum_{\lambda_1, \lambda_2, \lambda_3, \lambda_4 }
\norm{    |D|\angles{D}^{s-2} P_{\lambda_4} \left( \angles{D}^{-{s}}   P_{\lambda_1} u_1  \cdot  \angles{D}^{-{s}}   P_{\lambda_2} u_2  \cdot  \angles{D}^{-{s}}    P_{\lambda_3} u_3  \right)}_{L_{T,x}^2} \norm{ P_{\lambda_4} u_4  }_{ L_{T,x}^2  }
\\
&\lesssim
\underbrace{\sum_{
	 \lambda_1, \lambda_2, \lambda_3, \lambda_4 	}   C(\lambda) 
a_{\lambda_1} a_{\lambda_2} a_{\lambda_3} a_{\lambda_4}}_{:= \mathcal S},
\end{split}
\end{equation}
where 
$$
C(\lambda) = B(\lambda) \cdot \lambda_4 \angles{\lambda_4}^{s-2} 
\angles{\lambda_1}^{-s} \angles{\lambda_2}^{-s } \angles{\lambda_3}^{-s }  
$$
with $B(\lambda)$ as in \eqref{B-c1}-\eqref{B-c3}.

\vspace{2mm}
So \eqref{duality-biest11-c} reduces 
to proving 
\begin{equation}
\label{mathS}
\mathcal S \lesssim \prod_{j=1}^4 \|(a_{\lambda_j})\|_{l^2_{\lambda_j}}.
\end{equation}

\vspace{2mm}
By symmetry of our argument, we may assume 
$ \lambda_1 \le \lambda_2 \le \lambda_3$. 

If $\lambda_3 \lesssim 1$, then 
$$
C(\lambda) \lesssim  \lambda_4 ( \lambda_1 \lambda_2)^\frac n2  \lesssim  \lambda_4 ( \lambda_1 \lambda_2 \lambda_3)^\frac n3  .
$$
Then applying the Cauchy-Schwarz inequality in $\lambda_1, 
  \lambda_2, \lambda_3 $ and $\lambda_4$, to
   obtain
\begin{align*}
\mathcal S &\lesssim 
 \sum_{    \lambda_1,  \lambda_2, \lambda_3, \lambda_4 \lesssim 1 }   \lambda_4 ( \lambda_1 \lambda_2 \lambda_3)^\frac n3 
a_{\lambda_1} a_{\lambda_2} a_{\lambda_3} a_{\lambda_4} \lesssim\prod_{j=1}^4 \|(a_{\lambda_j})\|_{l^2_{\lambda_j}}.
\end{align*}
So, we assume $\lambda_3 \gg 1$.  We consider three cases:
\begin{enumerate}[(i)]
\item $ \lambda_1 \sim \lambda_2 \sim \lambda_3$,

\item $ \lambda_1 \le \lambda_2 \ll \lambda_3$,
\item $ \lambda_1 \ll \lambda_2 \sim \lambda_3$.
\end{enumerate}

\subsection{ \underline{(i): $ \lambda_1 \sim \lambda_2 \sim \lambda_3$}}

\subsubsection{ \underline{$ n=1$}}
In this case, we have
$$
C(\lambda) \sim  \lambda_4 \angles{\lambda_4}^{s-2} 
 \lambda_3^{ -1/2  -3s }   .
$$
If $\lambda_4 \lesssim 1$, then 
\begin{align*}
\mathcal S &\lesssim 
 \sum_{ \lambda_4 \lesssim 1, \   \lambda_1\sim  \lambda_2\sim\lambda_3\gg 1 }   \lambda_4 \lambda_3^{ -1/2  -3s} 
a_{\lambda_1} a_{\lambda_2} a_{\lambda_3} a_{\lambda_4} \lesssim\prod_{j=1}^4 \|(a_{\lambda_j})\|_{l^2_{\lambda_j}}
\end{align*}
provided $s>-1/6 $. 

 If $\lambda_4 \gg 1$, then 
\begin{align*}
\mathcal S  &\lesssim 
 \sum_{ \lambda_4 \gg 1, \   \lambda_1\sim  \lambda_2\sim\lambda_3\gg 1 }   \lambda_4^{s-1} \lambda_3^{ -1/2  -3s } 
a_{\lambda_1} a_{\lambda_2} a_{\lambda_3} a_{\lambda_4} \lesssim\prod_{j=1}^4 \|(a_{\lambda_j})\|_{l^2_{\lambda_j}}
\end{align*}
provided $-1/6 <s <1$. In the case $s\ge 1$, we can use the estimate $\lambda_4^{s-1} \lambda_3^{ -1/2  -3s } \lesssim \lambda_4^{-1} \lambda_3^{ -3/2  -2s } $ to get the desired estimate.

\subsubsection{ \underline{$ n\ge 2$}}

In this case, we have
$$
C(\lambda) \sim  \lambda_4 \angles{\lambda_4}^{s-2} 
 \lambda_3^{ n-3/2  -3s+ 3\delta }   .
$$
\begin{itemize}

\item
First, assume $\lambda_4 \lesssim 1$.
Then 
\begin{align*}
\mathcal S &\lesssim 
 \sum_{ \lambda_4 \lesssim 1, \   \lambda_1\sim  \lambda_2\sim\lambda_3\gg 1 }   \lambda_4 \lambda_3^{ n-3/2  -3s+ 3\delta } 
a_{\lambda_1} a_{\lambda_2} a_{\lambda_3} a_{\lambda_4} \lesssim\prod_{j=1}^4 \|(a_{\lambda_j})\|_{l^2_{\lambda_j}}
\end{align*}
provided $s>n/3-1/2 + (3/2)\delta$. 

\item Next, assume $\lambda_4 \gg 1$.

\begin{itemize}

\item
If $n\le 4$ and $n/3-1/2 +  (3/2)\delta<s<1$, then
\begin{align*}
\mathcal S  &\lesssim 
 \sum_{ \lambda_4 \gg 1, \   \lambda_1\sim  \lambda_2\sim\lambda_3\gg 1 }   \lambda_4^{s-1} \lambda_3^{ n-3/2  -3s+ 3\delta } 
a_{\lambda_1} a_{\lambda_2} a_{\lambda_3} a_{\lambda_4} \lesssim\prod_{j=1}^4 \|(a_{\lambda_j})\|_{l^2_{\lambda_j}}.
\end{align*}
\item If $n\le 4$ and $s\ge 1$, then
 \begin{align*}
\mathcal S &\lesssim 
 \sum_{ \lambda_4 \gg 1, \   \lambda_1\sim  \lambda_2\sim\lambda_3\gg 1 }    \lambda_4 ^{-\delta}     \lambda_3^{n-5/2  -2s+ 4\delta }  
a_{\lambda_1} a_{\lambda_2} a_{\lambda_3} a_{\lambda_4} \lesssim\prod_{j=1}^4 \|(a_{\lambda_j})\|_{l^2_{\lambda_j}}
\end{align*}
since $n-5/2  -2s+ 4\delta <-1/2 + 4\delta$.

\item  If  $n\ge 5$ and $ s > n/2-5/4+ 4\delta$ (which implies $s\ge 1$), then it is clear that the proceeding estimate that $\mathcal S$ is summable.
  \end{itemize}

    \end{itemize}

\subsection{ \underline{  (ii): $  \lambda_1 \le \lambda_2 \ll \lambda_3$}}

This implies $\lambda_3 \sim \lambda_4$.

\subsubsection{ \underline{$ n=1$}}
In this case, we have
$$
C(\lambda) \sim B(\lambda)  \angles{\lambda_1}^{-s}   \angles{\lambda_2}^{-s }\lambda_3^{-1 }  
$$

\begin{itemize}

\item
If $\lambda_1 \le \lambda_2 \lesssim 1$, then we can use the estimate
$
C(\lambda) \sim  (\lambda_1  \lambda_2) ^{ 1/2}      \lambda_3^{-1 } 
$
to sum up $\mathcal S$.

\item If $\lambda_1 \lesssim 1$, $\lambda_2 \gg 1$, then we can use the estimate
$
C(\lambda) \sim  (\lambda_1  \lambda_2) ^{ 1/2}     \lambda_2^{-s }  \lambda_3^{-1 } \lesssim \lambda_1^{ 1/2}     \lambda_2^{-1/2-s }  
$
to sum up $\mathcal S$ provided $s>-1/2$.

\item If $ 1 \ll \lambda_1 \le \lambda_2$, we can take $B(\lambda) \sim   (\lambda_2  \lambda_3) ^{ -1/4}    $, and hence $
C(\lambda) \sim  \lambda_1^{-3/4-s }  \lambda_2^{-3/4-s } .
$
So $\mathcal S$ is summable provided $s>-3/4$.

\end{itemize}

\subsubsection{ \underline{$ n\ge 2$}}
In this case, we have
$$
C(\lambda) \lesssim  \lambda_1 ^\frac n2   \angles{\lambda_1}^{-s}  \lambda_2 ^{\frac n2 - 1 + 2 \delta}  \angles{\lambda_2}^{-s }\lambda_3^{-\frac32 + \delta }  .
$$
\begin{itemize}

\item If $\lambda_1 \le \lambda_2 \lesssim 1$, then 
$
C(\lambda) \sim  \lambda_1 ^{ n/2}    \lambda_2 ^{ n/2 - 1 + 2 \delta}   \lambda_3^{-3/2 + \delta } .
$
Therefore, $\mathcal S$ is summable,
  and hence \eqref{mathS} holds. 

\item If $\lambda_1 \lesssim 1$, $\lambda_2 \gg 1$, then 
$
C(\lambda) \lesssim  \lambda_1 ^{n/2 }   \lambda_2 ^{n/2 - 5/2-s + 3 \delta},  
$
 and hence $\mathcal S$ is summable if $s>n/2-5/2+ 3\delta$.

\item If $ 1 \ll \lambda_1 \le \lambda_2$, then 
$
C(\lambda) \lesssim  \lambda_1 ^{n/2-s}    \lambda_2 ^{ n/2 - 5/2-s + 3 \delta},   
$
and hence $\mathcal S$ is summable provided $s>n/2$. On the other hand,  in the case $s\le n/2$, we have
$
C(\lambda) \lesssim  \lambda_1^{-\delta}    \lambda_2 ^{n - 5/2-2s + 4 \delta}  ,
$
and so $\mathcal S$ is summable provided $s>n/2-5/4+4 \delta$.
\end{itemize}

\subsection{ \underline{(iii): $\lambda_1 \ll \lambda_2 \sim \lambda_3 $}}

\subsubsection{ \underline{$ n=1$}}

\subsubsection*{ (a) \underline{Sub-case: $\lambda_4 \sim \lambda_3  $ }}

In this case, we have
$$
C(\lambda) \sim B(\lambda)  \angles{\lambda_1}^{-s}  \lambda_3^{-1-s }  .
$$

\begin{itemize}

\item
If $\lambda_1  \lesssim 1$, then we take $B(\lambda) \sim   \lambda_1^{1/2}  \lambda_3 ^{ -1/4}    $, and hence the estimate
$
C(\lambda) \sim  \lambda_1  ^{ 1/2}      \lambda_3^{-5/4 -s } 
$
can be used to sum up $\mathcal S$ if $s>-5/4$.

\item If $\lambda_1  \gg 1$, then then we take $B(\lambda) \sim   (\lambda_2 \lambda_3 )^{ -1/4}    $, and hence the estimate
$
C(\lambda) \lesssim  \lambda_1  ^{ -3/4 -s}      \lambda_3^{-3/4 -s } 
$
can be used to sum up $\mathcal S$ provided $s>-3/4$.

\end{itemize}

\subsubsection*{(b) \underline{Sub-case: $\lambda_4 \ll \lambda_3  $ }}
In this case, we have 
$$
C(\lambda) \sim  \lambda_4 \angles{\lambda_4}^{s-2}   \angles{\lambda_1}^{-s}  \lambda_3^{-1/2-2s } 
$$
where we chose $B(\lambda) \sim   (\lambda_2 \lambda_3 )^{ -1/4}  \sim  \lambda_3 ^{ -1/2}  $. 

\begin{itemize}

\item If $\lambda_4,  \lambda_1  \lesssim 1$, then we ignore $\angles{\lambda_4}^{s-2}$ and $\angles{\lambda_1}^{-s} $ to sum up $\mathcal S$ provided $s>-1/4$.  

\item If $\lambda_4 \lesssim 1$ and $\lambda_1 \gg 1$, then 
$
C(\lambda) \sim  \lambda_4 \lambda_1 ^{ -s}
  \lambda_3^{ -1/2-2s } \lesssim \lambda_4 \lambda_1 ^{ -1/6-s}
  \lambda_3^{ -1/3-2s }   ,
$
and hence $\mathcal S$ can be summed for  $s>-1/6$.

\item If $\lambda_4 \gg 1$ and $\lambda_1 \lesssim 1$, then 
$
C(\lambda) \sim  \lambda_4  ^{ s-1}
  \lambda_3^{ -1/2-2s } .
$
This can be used to sum up $\mathcal S$ if $s>-1/4$. 

\item Finally, if
$\lambda_4, \lambda_1 \gg  1$, then
$
C(\lambda) \sim  \lambda_4  ^{ s-1}  \lambda_1  ^{ -s}
  \lambda_3^{ -1/2-2s } \lesssim \lambda_4  ^{ s-1}  \lambda_1  ^{ -1/6-s}
  \lambda_3^{ -1/3-2s } .
$
and hence $\mathcal S$ can be summed easily for $s>-1/6$.

\end{itemize}

\vspace{2mm}

\subsubsection{ \underline{$ n\ge 2$}}

\subsubsection*{ \underline{Sub-case: $\lambda_4 \sim \lambda_3  $ }}

In this case, we have
$$
C(\lambda) \sim  \lambda_1 ^\frac n2   \angles{\lambda_1}^{-s} \lambda_3^{\frac n2-\frac52  -s+ 3\delta }  
$$

\begin{itemize}
\item 
If $\lambda_1 \lesssim 1$, then we can ignore $\angles{\lambda_1}^{-s}$, and therefore $\mathcal S$ is summable provided $s>n/2-5/2 + 3\delta$. 

\item If $\lambda_1 \gg 1$, then 
$
C(\lambda) \sim  \lambda_1 ^{n/2-s}    \lambda_3^{ n/2-5/2  -s+ 3\delta }  ,
$
and hence $\mathcal S$ is summable provided $s>n/2$.  On the other hand, if 
 $s\le n/2$, we have
$
C(\lambda) \sim  \lambda_1 ^{-\delta}    \lambda_3^{n-5/2  -2s+ 4\delta }  ,
$
and hence $\mathcal S$ is summable provided $s>n/2-5/4 + 4 \delta$. 
\end{itemize}

\subsubsection*{\underline{Sub-case: $\lambda_4 \ll \lambda_3  $ }}
In this case, we have 
$$
C(\lambda) \sim  \lambda_4 \angles{\lambda_4}^{s-2} \lambda_1 ^\frac n2  
\angles{\lambda_1}^{-s}   \lambda_3^{\frac n2-\frac32  -2s+ 3\delta }   .
$$

\begin{itemize}
\item
If $\lambda_4,  \lambda_1  \lesssim 1$, then we ignore $\angles{\lambda_4}^{s-2}$ and $\angles{\lambda_1}^{-s} $, and therefore $\mathcal S$ can be summed provided $s>(n-3)/4$. 

\item If $\lambda_4 \lesssim 1$ and $\lambda_1 \gg 1$, then 
$
C(\lambda) \sim  \lambda_4 \lambda_1 ^{n/2  -s}
  \lambda_3^{ n/2-3/2  -2s+ 3\delta }   ,
$
and hence $\mathcal S$ can be summed if $s>n/2$. 

On the other hand, if $s\le n/2$, then
$
C(\lambda) \lesssim   \lambda_4   \lambda_1^{-\delta}    \lambda_3^{n-3/2  -3s+ 4\delta }  .
$
Hence $\mathcal S$ can be summed for $s>n/3-1/2+( 4/3) \delta$.

\item If $\lambda_4 \gg 1$ and $\lambda_1 \lesssim 1$, then 
$
C(\lambda) \sim  \lambda_4^{s-1} \lambda_1 ^{n/2 }
  \lambda_3^{n/2-3/2  -2s+ 3\delta }  .
$
Now if $n\le 4$ and $(n-3)/4<s< 1$, then  $\mathcal S$ can be summed. In the case $n\le 4$ and $s\ge 1$, we have
$$
C(\lambda) \lesssim \lambda_4  ^{-\delta}   \lambda_1 ^{ n/2 } \lambda_3^{ n/2- 5/2  -s+ 4\delta }  \lesssim \lambda_4  ^{-\delta}   \lambda_1 ^{ n/2 } \lambda_3^{ -3/2 + 4\delta }  , 
$$
and hence $\mathcal S$ can be summed. 

In the case $n\ge 5$ and $s>n/2-5/4 +4\delta$ (which implies $s> 1$), we can use the estimate $
C(\lambda) \lesssim \lambda_4  ^{-\delta}   \lambda_1 ^{ n/2 } \lambda_3^{ n/2- 5/4  -s+ 4\delta } $ to sum up $\mathcal S$.

\item Finally, if $\lambda_4, \lambda_1 \gg  1$, then
$
C(\lambda) \sim  \lambda_4^{s-1} \lambda_1 ^{n/2 -s}
 \lambda_3^{ n/2-3/2  -2s+ 3\delta }   ,
$
and hence $\mathcal S$ can be summed easily for $s>n/2$. 

On the other hand, if
 $1\le s\le n/2$, then we use 
$
C(\lambda) \sim  \lambda_4 ^{-\delta }   \lambda_1^{-\delta}    \lambda_3^{n-5/2  -2s+ 5\delta }  ,
$
to sum up $\mathcal S$ provided
$s>n/2-5/4 + 5\delta $. 

If
 $s<1$, then we use 
$
C(\lambda) \sim  \lambda_4^{s-1}  \lambda_1^{-\delta}    \lambda_3^{n-3/2  -3s+ 4\delta }  
$
to sum up $\mathcal S$ provided $s>n/3-1/2 + 4\delta $. 

\end{itemize}

\vspace{2mm}

\textbf{ \underline{Conclusion:}} From the cases (i), (ii) and (iii), we conclude that the estimate \eqref{mathS} holds provided
$$
s > \max \left(n/3-1/2, \ n/2-5/2, \  n/2-5/4, \ n/4-3/4 \right) = \begin{cases}
& n/3-1/2, \qquad  for \quad 2\le n \le 4,
\\
& n/2-5/4, \qquad  for \quad  n \ge 5.
\end{cases}
$$

Similarly, in the case $n=1$, \eqref{mathS} holds provided
$$
s>-1/6.
$$

\section{Analyticity}
In this section, we show the analyticity of solutions stated in Theorem \ref{theo-1-anal} by using the space $H^{\sigma,2}(\R)$.

As we mentioned before, the local existence in $H^s(\R)$ implies that \eqref{nsB} is locally well-posed in $H^{\sigma,2}(\R)$. We first show this fact.

By applying $\cosh(\sigma|D|)$ to \eqref{contract-op} and taking the $H^2(\R)$-norm on both sides, we see from $ |D|^3 S_m(t)  \lesssim 1$ that
\[
\sup_{t\in[0,T]}(\norm{u(t)}_{H^{\sigma,2}(\R)}+\norm{u_t(t)}_{H^{\sigma, -1}(\R)})
\lesssim
\norm{u_0}_{H^{\sigma,2}(\R)}+\norm{u_1}_{H^{\sigma, -1}(\R)}
+\int_0^T\|u^3(s)\|_{H^{\sigma, -1}(\R)}d s
\]
for some $T>0$. Now if we set $U=e^{\sigma|D|}u$,   we have from the Plancherel equality and the Sobolev embedding that
\[
\begin{split}
	\norm{e^{\sigma|D|}\left(e^{-\sigma|D|}U\right)^3}_{L^2(\R)}&=
\norm{\int_{\xi=\xi_1+\xi_2+\xi_3}e^{\sigma(|\xi|-|\xi_1|-|\xi_2|-|\xi_3|)}\prod_{j=1}^3\hat{U}(\xi_j)\;d\xi_1d\xi_2 d\xi_3}_{L^2(\R)}\\
&
\leq \norm{V^3}_{L^2(\R)}\\
&
\lesssim \norm{V}_{L^6(\R)}^3\\&
\lesssim \norm{V}_{H^2(\R)}^3 
\lesssim \norm{U}_{H^2(\R)}^3,
\end{split}
\]
where $V=(|\hat{U}|)^\vee$. This means that
\[
\norm{u^3}_{G^{\sigma,0}(\R)}
\lesssim
\norm{u}_{G^{\sigma,2}(\R)} ^3,
\]
if $n\leq 4$. Thus, the local well-posedness of \eqref{nsB} in $H^{\sigma,2}(\R)\times H^{\sigma, -1}(\R)$ is deduced.

Now, for a local solution $u(t)$ of \eqref{nsB}, we define $v_\sigma(t)=\cosh(\sigma|D|)u(t)$ and the modified energy
\begin{equation}
	E_\sigma(t)
	=\frac12\int_{R}|(-\Delta)^{-\frac12} (v_\sigma)_t|^2+|v_\sigma|^2+\beta|\nabla v\sigma|^2+ |\Delta v_\sigma|^2 d x+\frac14\int_{R} |v_\sigma|^4d x
\end{equation} 
 It is straightforward from \eqref{nsB} to see that
 \[
 E_\sigma(t)
 =E_\sigma(0)+\int_0^t\int_{R}
 (v_\sigma)_t
 \left(    f(v_\sigma)-\cosh(\sigma|D|) f(\sech(\sigma|D|)v_\sigma)\right)dx,
 \]
where $f(v)=v^3$.

The following lemma provides the principal estimate in the proof of Theorem \ref{theo-1-anal}.
 \begin{lemma}\label{nonlin-est}
 	Let $f(v)=v^3$. We have for $(v_\sigma)_t\in L^2(\R)$ and $v_\sigma\in H^2(\R) $ that
 	\begin{equation}
 		\left| \int_{R}
 		(v_\sigma)_t
 		\left(    f(v_\sigma)-\cosh(\sigma|D|) f(\sech(\sigma|D|)v_\sigma)\right)dx\right|\lesssim
 		\sigma^2
 		\norm{v_\sigma}_{H^2(\R)}^3
 		\norm{|D|^{-1}(v_\sigma)_t}_{L^2(\R)}.
 	\end{equation}
 \end{lemma}
\begin{proof}
	First we note that
	\begin{eqnarray*}
		&&\int_{\R}
		(v_\sigma)_t
		\left(     v_\sigma^3 -\cosh(\sigma|D|)  (\sech(\sigma|D|)v_\sigma)^3\right)d x \\
		&&\quad
		\lesssim 
		\||D|^{-1}(v_\sigma)_t\|_{L^2(\R)}
		\left\||D|\left(  v_\sigma ^3-\cosh(\sigma|D|)  (\sech(\sigma|D|)v_\sigma)^3\right)\right\|_{L^2(\R)}\\
		&&\quad
		\lesssim \||D|^{-1}(v_\sigma)_t\|_{L^2(\R)}
		\underbrace{\left\|
			|\xi|\int_{\xi=\sum_{j=1}^3\xi_j}\left(1-\cosh(\sigma|\xi|)\prod_{j=1}^3\sech(\sigma|\xi_j|)\right)
			\prod_{j=1}^3\widehat{v_\sigma}(\xi_j) d \xi_1d\xi_2 d \xi_3
			\right\|_{L^2(\R)}}_{:= J}.
	\end{eqnarray*}
To estimate $J$ we need the following estimate from \cite[Lemma 3]{dmt}:
\begin{equation}\label{coshpest}
	\xi=\sum_{j=1}^p \xi_j   \quad \Rightarrow \quad \left|1 -  \cosh |\xi| \prod_{j=1}^p  \sech |\xi_j|   \right| \le 2^p \sum_{  j\neq k =1}^p   |\xi_j| |\xi_k|\qquad (2\leq p\in\N).
\end{equation}

By symmetry, we may assume $|\xi_1|\leq |\xi_2|\leq    |\xi_3|$, and hence $|\xi|\le 3 |\xi_3|$.
Consequently,  denoting
$w_\sigma=\mathcal{F}^{-1}_x(|\widehat{v_\sigma}|)$, we obtain
from the Plancherel identity, the H\"{o}lder inequality and Sobolev embedding that
\begin{align*}
	J & \lesssim \sigma^2 \left\|
	\int_{\xi=\sum_{j=1}^3\xi_j}
	 \widehat{w_\sigma}(\xi_1)  \cdot |\xi_{2}| \widehat{w_\sigma}(\xi_{2}) \cdot |\xi_3|^2\widehat{w_\sigma}(\xi_{3})\;  d \xi_1d\xi_2 d \xi_3
	\right\|_{L^2(\R)}
	\\
	& \lesssim \sigma^2 \left\| w_\sigma  \cdot |D| w_\sigma  \cdot |D|^2 w_\sigma 
	\right\|_{L^2(\R)}
	\\
	& \lesssim \sigma^2  \left\|   w_\sigma \right\| _{L^\infty (\R)  }   \left\|  |D| w_\sigma \right\|_{L^\infty (\R)}   \left\|  |D|^2 w_\sigma \right\|_{L^2(\R)}
	\\
	& \lesssim \sigma^2   \left\|   v_\sigma \right\|^3_{H^2(\R)}.
\end{align*}
\end{proof}

Now we are in a position to complete the proof of Theorem \ref{theo-1-anal}.

We have from Lemma \ref{nonlin-est} and the definition of $E_\sigma$ that
$$
E_\sigma(t) \le E_\sigma(0) + C T \sigma^2  \||D|^{-1}(v_\sigma)_t\|_{L^\infty_TL^2(\R)}  \left\|   v_\sigma \right\|^3_{ L^\infty_TH^2(\R)}
$$
for all $0\le t \le T$. Now, if $T$ is the local existence time, then we have 
\begin{align*}
	\left\|   v_\sigma \right\|_{ L^\infty_TH^2(\R)} + \||D|^{-1}(v_\sigma)_t\|_{L^\infty_TL^2(\R)} & \le C \left[
	\left\|  v_\sigma(\cdot, 0)\right\|_{ L^\infty_TH^2(\R)}  +  \||D|^{-1}  \partial_t v_\sigma(\cdot, 0)\|_{L^\infty_TL^2(\R)}  \right]
	\\
	&\le C \sqrt{ E_\sigma(0)}.
\end{align*}
Consequently,
$$
\sup_{0 \le t\le T}  E_\sigma(t) \le E_\sigma(0) + C T \sigma^2   \left[ E_{\sigma}(0)\right]^{\frac{p+1}2}  .
$$
This will eventually yield the asymptotic decay rate $\sigma \sim 1/ \sqrt t$ as $t \rightarrow + \infty$.

\section*{Acknowledgments}
A. E. is supported by the Nazarbayev University under Faculty Development Competitive Research Grants Program  for 2023-2025 (grant number 20122022FD4121). A. T. is supported by the Faculty Development Competitive Research Grants Program 2022-2024, Nazarbayev University: \emph{Nonlinear Partial Differential Equations in Material Science  (Ref. 11022021FD2929)}

\end{document}